\documentclass[12pt]{amsart}
 \usepackage{amsmath,amssymb,amsfonts,amsthm,amscd,textcomp,times,booktabs,mathrsfs}
   \usepackage[dvips]{graphicx}
  \usepackage{braket}
  \usepackage{color,float}
\usepackage[all]{xy}
\newtheorem{theorem}{Theorem}

\newtheorem{proposition}{Proposition}

\newtheorem{question}{Question}
\newtheorem{remark}{Remark}
\numberwithin{equation}{section} \numberwithin{theorem}{section}
\numberwithin{proposition}{section} \numberwithin{lemma}{section}
\numberwithin{claim}{section} \numberwithin{corollary}{section}

\newcounter{xpl}[section]

\numberwithin{equation}{section} \numberwithin{theorem}{section}

\newcommand{\mapright}[1]{\smash{\mathop{   \hbox to 0.7cm{\rightarrowfill}}
 \limits^{#1}}}

\newcommand{\Z}{\ensuremath{\mathbb{Z}}}

\newcommand{\C}{\ensuremath{\mathbb{C}}}

\newcommand{\reduce}[1]{\scalebox{1}{\ensuremath{#1}}}

\newcommand{\bracket}[1]{\ensuremath{\langle #1 \rangle}}
\newcommand{\norm}[1]{\ensuremath{\left| #1 \right|}}

 \DeclareMathOperator{\rank}{rank}

\DeclareMathOperator*{\smallwedge}{\reduce{\bigwedge}}

\def\C{\mathbb C}

\def\D{\Delta}

\def\P{\mathbb{P}}

\def\X{\mathscr{X}}
\begin{document}
\title[Diffeomorphism classes of the doubling Calabi-Yau threefolds]{Diffeomorphism classes of the doubling Calabi-Yau threefolds with Picard number two}

\author{Naoto Yotsutani}

\address{Kagawa University, Faculty of education, Mathematics, Saiwaicho $1$-$1$, Takamatsu, Kagawa,
$760$-$8522$, Japan}
\email{yotsutani.naoto@kagawa-u.ac.jp}

 \dedicatory{}
\maketitle
\noindent{\bfseries Abstract.}
Previously we constructed Calabi-Yau threefolds by a differential-geometric gluing method using Fano threefolds with their smooth anticanonical $K3$ divisors \cite{DY14}. 
In this paper, we further consider the diffeomorphism classes of the resulting Calabi-Yau threefolds (which are called the \emph{doubling Calabi-Yau threefolds}) 
starting from different pairs of Fano threefolds with Picard number one. Using the classifications of simply-connected $6$-manifolds in differential topology and the \emph{$\lambda$-invariant} introduced by Lee \cite{Lee20},
we prove that any two of the doubling Calabi-Yau threefolds with Picard number two are not diffeomorphic to each other when the underlying Fano threefolds are distinct families. 

\section{Introduction}\label{sec:Intro}
The purpose of this paper is to consider the diffeomorphism classes of Calabi-Yau $3$-folds with Picard number two constructed in our differential-geometrical gluing method. 
In \cite{DY14}, Doi and the author gave a differential-geometric construction (the \emph{doubling construction}) of Calabi-Yau $3$-folds starting from Fano $3$-folds with their smooth anticanonical $K3$ divisors 
(see Theorems $\ref{thm:Kov}$ and $\ref{thm:DY}$ for more details). Throughout this paper, we call the resulting Calabi-Yau $3$-folds obtained by the doubling construction the \emph{doubling Calabi-Yau $3$-folds}.

Calabi-Yau manifolds play the crucial role both in algebraic geometry and differential geometry. In particular, they form one of the building blocks in the classification of algebraic varieties up to birational isomorphism. 
We recall that a Calabi-Yau $n$-fold is an $n$-dimensional compact K\"ahler manifold $X$ whose canonical bundle is trivial and $H^i(X,\mathcal O_X)=0$ for $0<i<n$.
Then it is well-known that $X$ admits a Ricci-flat K\"ahler metric with holonomy $\mathrm{SU}(n)$. In the case of $n=1$, the only examples are genus $1$ curves and they are all diffeomorphic to each other.
In the case of $n=2$, Calabi-Yau surfaces are so-called $K3$ surfaces. Any two $K3$ surfaces are diffeomorphic as smooth $4$-manifolds (see Theorem $7.1.1$ in \cite{Huy16}).

In higher dimensions the classification of Calabi-Yau manifolds is hard to deal with and many problems are still open.
For example, it is an open problem whether or not the number of topological types of Calabi-Yau $3$-folds is bounded, while large amount of examples of Calabi-Yau $3$-folds have been discovered by many researchers 
in the context of Mirror symmetry. 
For a Calabi-Yau $3$-fold $X$, the \emph{topological type} of $X$ means the pair of Hodge numbers $(h^{1,1}(X), h^{2,1}(X))$.
It is known that if two Calabi-Yau $3$-folds $X_1$ and $X_2$ have the different topological types, they are not homeomorphic.
By contraposition, if two Calabi-Yau $3$-folds $X_1$ and $X_2$ homeomorphic, then they must have the same topological type.
It is an interesting question to ask whether the converse is also true, that is, if two Calabi-Yau $3$-folds $X_1$ and $X_2$ have the same topological type, are they homeomorphic too?
We prove that the answer of this question is {\emph{no}} in general (see Theorem $\ref{thm:main}$).

From a differential-geometric point of view, it is natural to classify Calabi-Yau $3$-folds up to diffeomorphism classes. For this purpose we use the classification of closed, oriented simply-connected $6$-manifolds by Wall, Jupp and Zhubr (see Section $\ref{sec:Diffeo}$ and the website of the Manifold Atlas Project, $6$-manifolds: $1$-connected \cite{MAP} for a good overview which includes further references). 
In a ward, two closed simply-connected $6$-manifolds that are homeomorphic are necessarily diffeomorphic. Classification of simply-connected $6$-manifolds with torsion free homology up to diffeomorphism classes is determined by the {\emph{basic invariants}} in \cite[p.$300$]{OkVa95}. The essential invariants here are the cubic intersection forms on $H^2(X,\Z)$ and the Chern classes.

We recall that two compact complex manifolds $X_1$ and $X_2$ are said to be {\emph{deformation equivalent}} if there is a smooth proper holomorphic map $\varpi: \X \to B$  satisfying the following conditions:
\begin{enumerate}
\item[(i)] The total space $\X$ and the base space $B$ are connected.
\item[(ii)] There exist two points $t_1, t_2 \in B$ such that $\varpi^{-1}(t_i) \cong X_i$ $(i=1,2)$.
\end{enumerate} 
Remark that deformation equivalent Calabi-Yau $3$-folds are in particular diffeomorphic although the converse is not true in general.
In fact, there exist diffeomorphic Calabi-Yau $3$-folds which are not deformation equivalent. See \cite{Gr97} and \cite[p.$420$]{Th99}.

There are several ways to construct Calabi-Yau $3$-folds in the study of algebraic and complex geometry.
The large amount of examples of Calabi-Yau $3$-folds are brought by Batyrev construction in the context of mirror symmetry \cite{Bat94}. 
Starting from a $4$-dimensional reflexive polytope $\Delta$, we consider the associated Gorenstein toric Fano variety $\P_\D$ and take a generic anticanonical section $\overline{X}$ in $\P_\D$.
By taking a crepant resolution of $\overline{X}$, we obtain a Calabi-Yau $3$-fold whose Hodge numbers are explicitly calculated by combinatorial data of $\D$.
We remark that $4$-dimensional reflexive polytopes are already classified and there are $473,800,776$ isomorphism classes.
Indeed, there exist (at least) $30,108$ topologically distinct Calabi-yau $3$-folds obtained by Batyrev's construction \cite{KS00}.
This means that most of the values of their Hodge numbers  stay in the same region for some reason.
Hence it is interesting to ask the number of diffeomorphism types of the resulting Calabi-Yau $3$-folds (with the aid of computer programs), by calculating not only their Hodge numbers but also
the cubic forms and the $\lambda$-invariants (see, Section \ref{sec:CubicForm}).

Another construction was developed by Kawamata and Namikawa \cite{KN94}. They investigated log deformation theory of normal crossing varieties and constructed examples of Calabi-Yau $3$-folds by smoothing
simple normal crossing varieties. Later on, Lee investigated log deformations of simple normal crossing varieties of two 
irreducible components which are so-called Tyurin degenerations (see, Definition V.$1$ in \cite{LeeThesis}). In particular, he gave explicit descriptions of the Picard groups, Hodge numbers and the second Chern classes of $X_t$, where $X_t=\varpi^{-1}(t)$ $(t\neq 0)$ is a general fiber of Tyurin degenerations $\varpi:\mathscr{X}\to B$ of a simple normal crossing variety $X_0=\varpi^{-1}(0)=Y_1\cup Y_2$. See \cite{Lee10} for more details.


Differential-geometric counterpart of Lee's construction was studied by Doi and the author in \cite{DY14}.
We gave a differential-geometric construction ({{the doubling construction}}) of Calabi-Yau $3$-folds by gluing two asymptotically cylindrical Ricci-flat K\"ahler manifolds along their cylindrical ends 
in appropriate conditions. Also we computed the set of all Hodge numbers $(h^{1,1},h^{2,1})$ (i.e. the topological types) of Calabi-Yau $3$-folds obtained by the doubling construction. 
See Figure $1$ and Section $\ref{sec:DY}$ for more details.
\begin{figure}[h]
\centering
\includegraphics[width=\hsize ,clip]{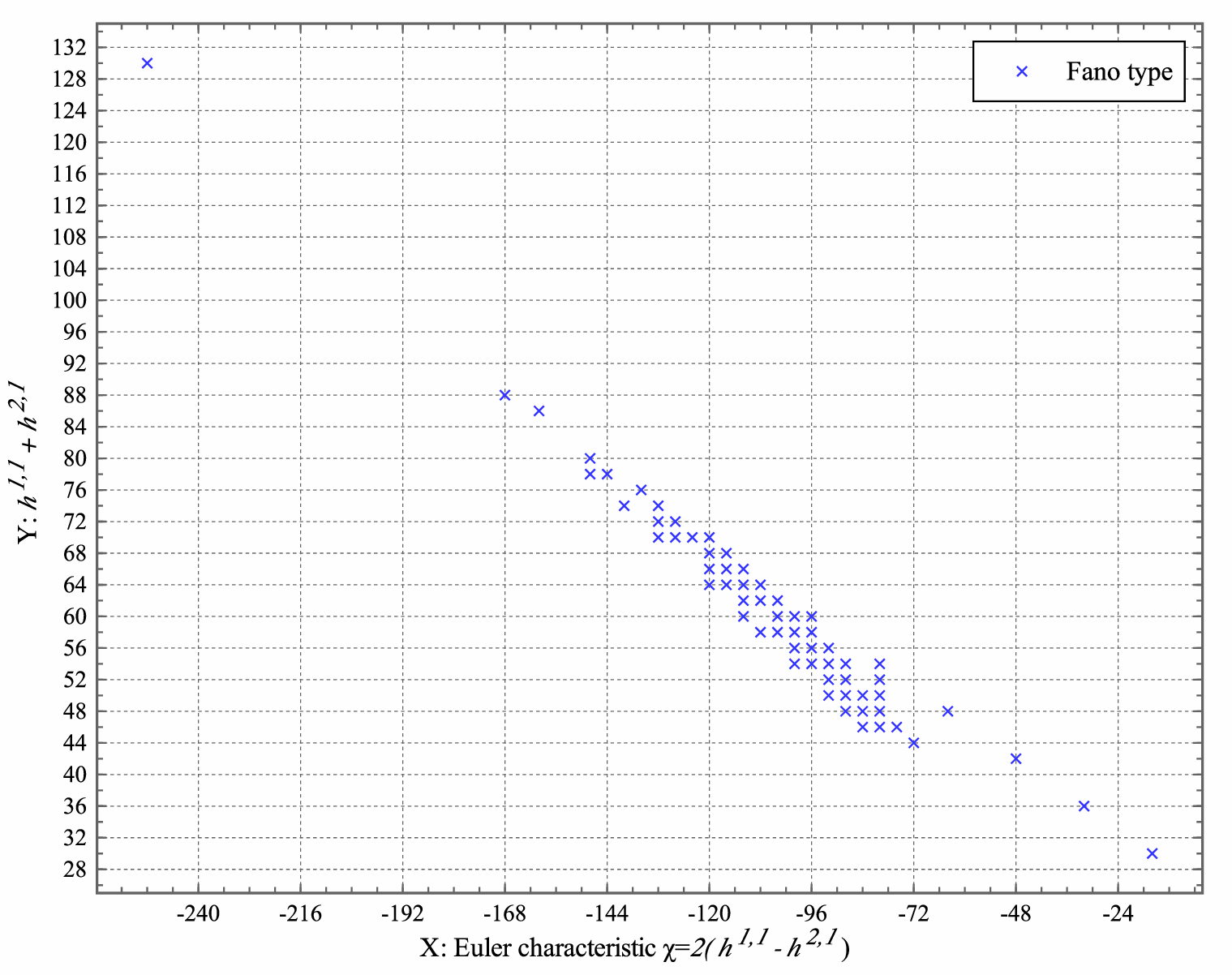}
\caption{All resulting Calabi-Yau $3$-folds using Fano $3$-folds}
\end{figure}
Consequently we found $61$ distinct $(h^{1,1}(M),h^{2,1}(M))$ for the doubling Calabi-Yau $3$-folds $M$ (Proposition $\ref{prop:DY}$).  
As ingredients of the doubling construction, we used {\emph{admissbile pairs}} $(Y,D)$ consisting of a three dimensional compact K\"ahler manifold $Y$ and a smooth anticanonical divisor $D$ on $Y$.
We can construct admissible pairs from Fano $3$-folds as follows.
Beginning with a Fano $3$-fold $V$ with a smooth anticanonical $K3$ divisor $D$, one can show that if we blow-up $V$ along a curve representing $D\cdot D$ to obtain $Y$,
then $Y$ has an anticanonical divisor isomorphic to $D$ (denoted by $D$ again) with the holomorphic normal bundle $\mathcal N_{Y/D}$ trivial. In particular, $(Y,D)$ is an admissible pair (Theorem $\ref{thm:Kov}$).
Note that $Y$ itself is not a Fano $3$-fold in this construction.
According to the classification of smooth Fano $3$-folds
\cite{Is77, MM81, MM03}, there are $105$ algebraic families with Picard numbers $1\leqslant \rho(V) \leqslant 10$. 
It is remarkable that $61$ ($=$ the total number of topological types of the doubling Calabi-Yau $3$-folds) is much less than $105$ families of Fano 
$3$-folds. 
This means that many of the doubling Calabi-Yau $3$-folds $M$ 
have the same Hodge numbers $(h^{1,1}(M), h^{2,1}(M))$ even though we use distinct Fano $3$-folds as ingredients.
More precisely, as listed in Table $\ref{table:CY3Pic2}$, there are $8$ doubling Calabi-Yau $3$-folds $M$ with Picard number two which have the same Hodge numbers $(h^{1,1}(M), h^{2,1}(M))$ although we started from
distinct $17$ families of Fano $3$-folds with Picard number one. 
These $8$ overlapping Hodge numbers $(h^{1,1}(M), h^{2,1}(M))$ are listed with symbol $\checkmark$ on the table.
\begin{table}\label{table:CY3Pic2}
\centering
    \caption{The doubling Calabi-Yau $3$-folds with Picard number two and the underlying Fano $3$-folds with Picard number one}
 \begin{tabular}{r||c|c|c}
 \vspace{-0.23cm}                                                               &                    &                              &      \\
         \vspace{-0.23cm}   ID in \cite{Fano}                                     &  $-K_V^3$   &     $h^{1,2}(V)$          &  $(h^{1,1}(M),h^{2,1}(M))$    \\
                                                                                 &                     &                            &     \\
 \hline
$1$-$1$     \hspace{0.5cm}                                                             &  $2$             & $52$              & $(2,128)$    \\
 $1$-$2$       \hspace{0.5cm}                                                                         &   $4$            & $30$             & $\checkmark~ (2,86)$ \hspace{0.34cm}   \\
$1$-$3$               \hspace{0.5cm}                                                         &   $6$             & $20$             & $(2,68)$    \\
 $1$-$4$                \hspace{0.5cm}                                                             &  $8$             & $14$             & $\checkmark~ (2,58)$  \hspace{0.34cm}   \\
$1$-$5$                \hspace{0.5cm}                                                                 &  $10$            & $10$              & $(2,52)$    \\
$1$-$6$                  \hspace{0.5cm}                                                             &  $12$             & $7$              & $(2,48)$    \\
$1$-$7$                  \hspace{0.5cm}                                                       &   $14$            & $5$             & $(2,46)$    \\
 $1$-$8$                    \hspace{0.5cm}                                                          &   $16$             & $3$             & $\checkmark~ (2,44)$ \hspace{0.34cm}    \\
 $1$-$9$                    \hspace{0.5cm}                                                        &  $18$             & $2$             & $\checkmark~ (2,44)$  \hspace{0.34cm}   \\
 $1$-$10$                  \hspace{0.5cm}                                                           &  $22$            & $0$              & $\checkmark~ (2,44)$ \hspace{0.34cm}     \\
$1$-$11$                    \hspace{0.5cm}                                                         &  $8$             & $21$              & $(2,72)$    \\
 $1$-$12$                   \hspace{0.5cm}                                                 &   $16$            & $10$             & $\checkmark~ (2,58)$ \hspace{0.34cm}    \\
$1$-$13$                    \hspace{0.5cm}                                                      &   $24$             & $5$             & $(2,56)$    \\
 $1$-$14$                  \hspace{0.5cm}                                                        &  $32$             & $2$             & $\checkmark~ (2,58)$ \hspace{0.34cm}    \\
$1$-$15$                  \hspace{0.5cm}                                                    &  $40$            & $0$              & $(2,62)$    \\
$1$-$16$                   \hspace{0.5cm}                                                &  $54$             & $0$              & $(2,76)$    \\
 $1$-$17$                   \hspace{0.5cm}                                              &  $64$             & $0$              & $\checkmark~ (2,86)$ \hspace{0.34cm}     \\
  \end{tabular}
  \end{table}
Remark that all of the underlying Fano $3$-folds $V$ have the same Picard number one because the cohomology of $M$ is determined by the cohomology of $V$ in the doubling construction. For more details, see Proposition $5.2$ in \cite{DY14}. We also use the database ID in \cite{Fano} to identify the underlying Fano $3$-folds $V$. We remark that there are other examples of such doubling Calabi-Yau $3$-folds $M$ having the same Hodge numbers when
the underlying Fano $3$-folds $V$ have the Picard number $\rho(V)\geqslant 2$. See Tables $6.2$--$6.5$ in \cite{DY14} for the details.
Hence it is natural to ask whether they are further diffeomorphic or not.

\begin{question}\label{Q:CY3-folds}
Let $V_1$ and $V_2$ be distinct Fano $3$-folds among $105$ families. Let $M_1$ and $M_2$ be the doubling Calabi-Yau
$3$-folds having the same Hodge numbers $(h^{1,1}, h^{2,1})$. Are they also diffeomorphic? 
\end{question}
We prove that the answer to Question $\ref{Q:CY3-folds}$ is negative if the underlying Fano $3$-folds have Picard number one.
\begin{theorem}\label{thm:main}
Let $V_1$ and $V_2$ be two distinct Fano $3$-folds with Picard number one.
Then the corresponding doubling Calabi-Yau $3$-folds $M_1$ and $M_2$ are not diffeomorphic. 
\end{theorem}
Section $\ref{sec:CubicForm}$ gives the proof of Theorem $\ref{thm:main}$.
As we mentioned in the above, in the proof of Theorem $\ref{thm:main}$, 
we use the classification results of closed, oriented, simply-connected $6$-manifolds with torsion free homology. 
One of the essential invariant in our proof is the cubic intersection form $\mu$ (see $\eqref{Invariants}$ for the definition) on $H^2(M,\mathbb Z)$.
We explicitly compute the values of the cubic forms $\mu$ for $8$ doubling Calabi-Yau $3$-folds with symbol $\checkmark$ in Table $\ref{table:CY3Pic2}$.
Taking specific Calabi-Yau $3$-folds $M$ and $M'$, we prove that each of them has the different values of cubic forms in $H^6(M,\Z)$ (resp. in $H^6(M',\Z)$).
However, we still can not conclude that there is no isomorphism between their cubic forms which sends $c_2(M)$ to $c_2(M')$. This is because the values of cubic forms may change if we use different basis 
of $H^2(M,\Z)$ (resp. $H^2(M',\mathbb Z)$). This is similar to the situation that if we use different basis of an inner product vector space $\mathbb V$, then we have the different values of inner products between
elements in the basis. Therefore, after calculating cubic forms for some specific basis of $H^2(M,\Z)$ and $H^2(M',\Z)$, we still need to show that the cubic forms $\mu$ and $\mu'$ are not isomorphic.
In order to solve this problem, we use the {\emph{$\lambda$-invariant}} which has been introduced in \cite{Lee20} to distinguish $M$ and $M'$ up to diffeomorphism.
We note that the $\lambda$-invariant can be defined for any compact complex $3$-fold with $h^2(M):=\rank (H^2(M,\Z))=2$ whose second Chern class $c_2(M)$ is not zero.
All of the values of cubic forms and the $\lambda$-invariants are listed in Table $\ref{table:invariants}$ which will be used in the proof of Theorem $\ref{thm:main}$.
\begin{table}
   \caption{Cubic forms and the $\lambda$-invariants of $8$ overlapping doubling Calabi-Yau $3$-folds with Picard number two.}
\begin{tabular}{cccccc}\toprule \label{table:invariants}
&Description of & & && \vspace{-0.25cm} \\ 
ID in \cite{Fano}& \qquad& $(h^{1,1}(M),h^{2,1}(M))$ & $(e_1^3, e_1^2e_2, e_1e_2^2, e_2^3)$ & $\lambda(M)$ \vspace{-0.25cm} \\ 
&Fano $3$-fold & \qquad&&  &  \\ \midrule
$1$-$2$ &$V(4)\subset \C P^4 $  &$(2,86)$ & $(8,4,0,0)$ & $540$ \\ 
$1$-$17$ &$\C P^3$ & $(2,86)$ & $(2,4,0,0)$ & $4320$ \\ \hline
$1$-$8$ &$V_{16}\subset \C P^{10}$  &$(2,44)$ & $(32,16,-240,-708)$ & $1672224$ \\ 
$1$-$9$ &$V_{18}\subset \C P^{11}$  &$(2,44)$ & $(36,18,-306,-904)$ & $5529560$ \\ 
$1$-$10$ &$V_{22}\subset \C P^{13}$  &$(2,44)$ & $(44,22,-462,-1358)$ & $122507896$ \\ \hline
$1$-$4$ &$V(2,2,2)\subset \C P^6$  &$(2,58)$ & $(16,8,-56,-164)$ & $1920$ \\ 
$1$-$12$ &$V(4)\subset \C P(1^4,2)$  &$(2,58)$ & $(32,32,-192,-1364)$ & $208516$ \\ 
$1$-$14$ &$V(2,2)\subset \C P^5$  &$(2,58)$ & $(64,64,-896,-5796)$ & $3440828$ \\ \bottomrule
\end{tabular}
\end{table}
In Table $\ref{table:invariants}$, $e_i$ denotes the basis of $H^2(M,\Z)$. We refer the reader to see Section $\ref{sec:CubicForm}$ and \cite{Y22} for more details on the computations.

The author expects that Theorem $\ref{thm:main}$ is also true for the doubling Calabi-Yau $3$-folds with $\rho(M)\geqslant 3$. However, there are two difficulties in proving the assertion.
Firstly, it is hard to specify the generators of $H^2(M,\Z)$ for the doubling Calabi-Yau $3$-folds with $\rho(M)\geqslant 3$ and it requires a fairly number of computations. Secondly, we struggle with more difficult problem:
when $h^2(M)$ is lager than $2$, the $\lambda$-invariant can not be defined. Hence we have to consider more intrinsic invariants of cubic forms (for example, the Arnold $S$-invariants of cubic forms for compact complex
$3$-folds with $h^2(M)=3$).

This paper is organized as follows. Section $\ref{sec:CYconstruction}$ is devoted to reviewing some essential background.
In Section $\ref{sec:Diffeo}$, we recall Wall's theorem on the classification of closed, oriented, simply-connected smooth $6$-manifold with torsion free homology up to diffeomorphism.
For our purpose, we shall use Jupp's theorem on the invariants that determine the diffeomorphism classes of simply-connected compact complex $3$-folds. We also introduce the $\lambda$-invariant which can be defined on any
compact complex $3$-fold $M$ with $h^2(M)=2$. 
In the last section, we prove Theorem $\ref{thm:main}$ by calculating (i) the values of the cubic forms and (ii) the values of the $\lambda$-invariant with respect to some specific basis
of $H^2(M,\Z)$ for each doubling Calabi-Yau $3$-fold $M$ with symbol $\checkmark$ in Table $\ref{table:CY3Pic2}$. More detailed and further computations can be found in the paper \cite{Y22}.

\vskip 4pt

\noindent{\bf{Acknowledgement.}}
I would like to thank Professor Nam-Hoon Lee for answering patiently my questions on cubic forms and the $\lambda$-invariants of Calabi-Yau $3$-folds.
I also thank Professor Ryushi Goto and Dr. Mamoru Doi for asking an interesting question on diffeomorphism classes of Calabi-Yau $3$-folds (Question $1$).
This work was partially supported by JSPS KAKENHI Grant Number $18K13406$ and $22K03316$.
\vskip 4pt

\noindent{\bf{Declaration.}}
The author has no competing interests to declare that are relevant to the content of this article.

\section{Calabi-Yau constructions}\label{sec:CYconstruction}
\subsection{Algebro-geometric approach}
The smoothing problem of normal crossing variety was pioneered by R. Friedman in the early $1980$'s \cite{Fr83}.
He proved a smoothing theorem for normal crossing complex surfaces and showed that any $d$-semistable $K3$ surface $X$ has a family of
smoothings $\varpi: \X \to \D \subset \C$ with $K_\X=\mathcal O_\X$, where $\X$ is a $3$-dimensional complex manifold and $\varpi$ is a
proper map between $\X$ and a domain $\D$ in $\C$. This result was generalized by Kawamata and Namikawa to higher dimensional case
\cite{KN94}. They used $T^1$-lifting property for proving the unobstructedness of the deformation space of smooth Calabi-Yau $n$-folds
\cite{Ra92, Kaw92}. In particular we have to require the cohomological condition $H^{n-1}(X,\mathcal O_X)=0$ in order to carry out
$T^1$-lifiting property. Recall that a complex simple normal crossing (SNC for short) variety
$X=\bigcup_{i=1}^N Y_i$ is {\emph{ $d$-semistable}} if
\[
\left( \bigoplus_{i=1}^N \mathcal{I}_{Y_i}/\mathcal{I}_{Y_i}\mathcal{I}_{D} \right)^{\!\!*}
\simeq \, \mathcal O_D
\]
for $D=\mathrm{Sing}X$, where $\mathcal{I}_{Y_i}$ and $\mathcal I_D$ are the ideal sheaves of $X_i$ and $D$ in $X$ respectively.
If $N=2$ $\bigl(\text{i.e. }\, X=Y_1\cup Y_2 \bigr)$, then $X$ is $d$-semistable if and only if
\begin{equation}\label{condi:d-s.s}
\mathcal N_{D/Y_1}\otimes \mathcal N_{D/Y_2}\simeq \mathcal O_D
\end{equation}
where $\mathcal N_{D/Y_i}$ is the normal bundle of $D$ in $Y_i$. A degeneration $\varpi: \X \to \D$ of $X$ is {\emph{semistable}} if
the total space $\X$ is smooth and the central fiber $X_0:=\varpi^{-1}(0)\cong X$ is K\"ahler.
It is known that an SNC variety $X$ is $d$-semistable if $X$ is the central fiber in a semistable degeneration (see, Corollary $1.12$ in \cite{Fr83}).
Although the converse is not true in general, in the case of Calabi-Yau manifolds,
the theorem of Kawamata and Namikawa states the following.
\begin{theorem}[Kawamata-Namikawa]\label{thm:KN}
Let $  X=\bigcup_{i=1}^N Y_i$ be a proper simple normal crossing (SNC) variety with $\dim_{\C} X=n \geqslant 3$ satisfying the following conditions:
\begin{enumerate}
\item the dualizing sheaf is trivial, that is, $\omega_X \simeq \mathcal O_X$.
\item $H^{n-1}(X, \mathcal O_X)=0$ and $H^{n-2}(Y_i,\mathcal O_{Y_i})=0$ for any $i$.
\item $X$ is $d$-semistable.
\end{enumerate}
Then $X$ has a family of smoothings $\varpi:\X \to \Delta$ with the smooth total space.
\end{theorem}

\begin{remark}\rm
Originally it is assumed that $X$ is K\"ahler in \cite[Theorem $4.2$]{KN94}.
However we only need to assume that $X$ is a proper SNC variety. See Remark $2.9$ in \cite{HS19} for more details.
In particular, Hashimoto and Sano constructed infinitely many topological types (and hence diffeomorphic classes) of non-K\"ahler compact complex
$3$-fold with trivial canonical bundle (see, Theorem $1.2$ in \cite{HS19}).
See also \cite{FFR19} where the authors proved Theorem $\ref{thm:KN}$ under mild assumptions instead of $d$-semistability (3) and cohomological conditions (2).
Recently we gave a differential-geometric interpretation of Friedman's smoothability result on $d$-semistable $K3$ surfaces \cite[Theorem $5.10$]{Fr83} and extended it to cases where not all irreducible components of $X$ are K\"ahlerian or
$H^1(X,\mathcal O_X)$ does not vanish. We refer the reader \cite{DY22} for more details.
\end{remark}

Later on, Lee investigated log  deformations of SNC varieties consisting of two irreducible components. 
Moreover, he described the Picard group of a Calabi-Yau $n$-fold with $n\geqslant 3$ obtained by Theorem $\ref{thm:KN}$. 

\vskip 4pt

For our purpose, we focus on complex three dimensional case and the case where the central fiber has only two components, that is,
$X=Y_1\cup Y_2$. Let $D:=Y_1\cap Y_2$. Then $X$ is projective (and so is K\"ahler) if and only if there are ample divisors $H_1$
on $Y_1$ and $H_2$ on $Y_2$ such that $H_1\big|_D$ is linearly equivalent to $H_2\big|_D$.
In particular, the above conditions $(1)$--$(3)$ in Theorem $\ref{thm:KN}$ are equivalent to the following conditions: 
\begin{enumerate}
\item[$(1)'$] $D$ is an anticanonical divisor on each $Y_i$: $D\in |-K_{Y_i}|$.
\item[$(2)'$] $H^1(Y_i, \mathcal O_{Y_i})=H^1(D, \mathcal O_D)=0$.
\end{enumerate}
By condition $(1)'$ and the adjunction formula, we have 
\[
K_D=(K_X\otimes[D])|_D=[-D]|_D\otimes[D]|_D \simeq \mathcal O_D.
\]
Thus condition $(2)'$ yields that $D$ is a $K3$ surface. As we have seen in $\eqref{condi:d-s.s}$,
we require 
\begin{enumerate}
\item[$(3)'$] $\mathcal N_{D/Y_1}\otimes \mathcal N_{D/Y_2}\simeq \mathcal O_D$
\end{enumerate}
for $X$ to be $d$-semistable. Then by Theorem $\ref{thm:KN}$,
$X$ is smoothable to a complex $3$-fold $M_t:=\varpi^{-1}(t)$ with trivial canonical bundle
and $H^i(M_t, \mathcal O_{M_t})=H^0(M_t, \Omega_{M_t}^i)=0$ for $i=1,2$. 
Especially $M_t$ is a Calabi-Yau $3$-fold.
Lee gave the following description of $H^2(M_t,\Z)$ in terms of the normal crossing central fiber.
\begin{theorem}[Corollary $8.2$ in \cite{Lee10}]\label{thm:NamHoon}
Let $M_t$ be the Calabi-Yau $3$-fold obtained in the above and $k=\rank \left(H^2(M_t, \Z ) \right)$. 
Assume that there are some elements $\alpha_1, \ldots , \alpha_k$ in
\[
G:=\set{(\ell, \ell')\in H^2(Y_1,\Z)\times H^2(Y_2,\Z) | i_1^*(\ell)=i_2^*(\ell')}
\]
and $\beta_1, \ldots \beta_k$ in 
\[
\set{(\tilde \ell, \tilde \ell')\in H^4(Y_1,\Z)\times H^4(Y_2,\Z) | i_1^*(\tilde \ell)=i_2^*(\tilde \ell')}
\]
such that the $k\times k$ matrix $(\alpha_i \cdot \beta_j)$ is unimodular, where $i_a: D\to Y_a$ $(a=1,2)$ is the inclusion map.
Then
\begin{equation}\label{eq:2ndCohomology}
H^2(M_t, \Z) \cong G/\!\braket{D, -D}
\end{equation}
up to torsion with the cup products being preserved.
\end{theorem}

\subsection{Differential-geometric approach}\label{sec:DY}
In \cite{DY14}, we gave a differential-geometric construction of Calabi-Yau $3$-folds building upon the work of Kovalev \cite{Kov03}
and Kovalev-Lee \cite{KL11}. Ingredients in our construction are admissible pairs.
Here an admissible pair $(Y,D)$ consists of a $3$-dimensional compact 
K\"ahler manifold $Y$ and a smooth anticanonical $K3$ divisor $D$. 
More precisely $(Y, D)$ is said to be an admissible pair if the following conditions hold:
\begin{enumerate}
\item[(a)] $Y$ is a $3$-dimensional compact K\"ahler manifold.
\item[(b)] $D$ is a smooth anticanonical divisor on $Y$.
\item[(c)] The normal bundle $\mathcal N_{Y/D}$ is trivial.
\item[(d)] $Y$ and $Y\setminus D$ are simply-connected.
\end{enumerate}
Moreover one can construct admissible pairs (which are called {\emph{of Fano type}}) using a Fano $3$-fold with a smooth
anticanonical $K3$ divisor.
\begin{theorem}[Proposition $6.42$ in \cite{Kov03}]\label{thm:Kov}
Let $V$ be a Fano $3$-fold, $D\in|-K_V|$ a $K3$ surface, and let $C$ be a smooth curve in $D$ representing the self-intersection
class of $D\cdot D$. Let $\varphi: Y=\mathrm{Bl}_C(V)\dasharrow V$ be the blow-up of $V$ along the curve $C$.
Taking the proper transform of $D$ under the blow-up $\varphi$, we still denote it by $D$. Then $(Y,D)$ is an admissible pair.
\end{theorem}

Let $(Y,D)$ be an admissible pair. 
Let $\mathcal N=\mathcal N_{D/Y}$ denote the normal bundle.
Then there exists a Calabi-Yau structure on $Y\setminus D$ asymptotic to a cylindrical Calabi-Yau structure on $\mathcal N\setminus D$
due to the work of \cite{HHN15}. Hence one can obtain a compact manifold $M$ by gluing two copies of $Y\setminus D$ along their cylindrical ends. Furthermore we have the following.
\begin{theorem}[Corollary $3.12$ in \cite{DY14}]\label{thm:DY}
Let $(Y,D)$ be an admissible pair of Fano type with $\dim_\C Y=3$. Then we can construct the Calabi-Yau $3$-fold by gluing two copies
of $Y\setminus D$ along their cylindrical ends. 
\end{theorem}
We call the resulting manifold $M$ the {\em{doubling}} Calabi-Yau $3$-fold. 
Combining Theorems $\ref{thm:Kov}$ and $\ref{thm:DY}$, we can construct Calabi-Yau $3$-folds systematically
 starting from Fano $3$-folds.
 \begin{proposition}\label{prop:DY}
 Let $M$ be the doubling Calabi-Yau $3$-fold obtained by an admissible pair of Fano type.
 Then the number of distinct topological types of $(h^{1,1}(M), h^{2,1}(M))$ is $61$.
 \end{proposition} 
\begin{proof}
See Figure $1$ and Tables $6.1$--$6.5$ in \cite{DY14}.
\end{proof}

\section{Diffeomorphism classes of three dimensional Calabi-Yau manifolds}\label{sec:Diffeo}
In the rest of this paper, we shall consider the diffeomorphism classes of $3$-dimensional Calabi-Yau manifolds.
Hence our main concern is to distinguish given Calabi-Yau $3$-folds by their diffeomorphism type.
A good way to achieve this is to use the classification of closed simply-connected $6$-manifolds which is very complete.
Remark that if $M$ and $N$ are homotopy equivalent through a map which preserves the characteristic classes of the tangent bundle,
then they are diffeomorphic. This technique also works for $5$-manifolds, but it is false in real dimension $4$ and $7$.
In a word, two closed simply-connected smooth $6$-manifolds that are homeomorphic are necessarily diffeomorphic.
The spin case is due to C.T.C Wall \cite{Wa66} where he described the invariants that determine the diffeomorphism classes of simply-connected,
spin, oriented, closed $6$-manifold with torsion free homology. 
The following characterization is due to Jupp \cite{J73}, that is,
we can classify any compact simply-connected complex $3$-folds $M$ up to diffeomorphism by 

\begin{enumerate}
\item[(i)] a symmetric trilinear form (cubic form)
\begin{align} \label{Invariants}
\begin{split}
\mu: H^2(M,\Z)\otimes H^2(M,\Z)\otimes H^2(M,\Z) & ~~ \longrightarrow ~~H^6(M,\Z)\cong \Z \\
(x,y,z)& ~~ \longmapsto ~~ x\cup y\cup z
\end{split}
\end{align}
where $\cup$ denotes the cup product of differential forms, and
\vskip 4pt

\item[(ii)] the Chern classes of $M$.
\end{enumerate}
See \cite{OkVa95, KW14, BI15} for further details. 
In Section $\ref{sec:CubicForm}$, we pick up $8$ doubling Calabi-Yau $3$-folds having the overlapping Hodge numbers $(h^{1,1},h^{2,1})$ from Table $\ref{table:CY3Pic2}$ and prove that they are not diffeomorphic.
Firstly we calculate the values of the cubic forms of the doubling Calabi-Yau $3$-folds and confirm that the values $\eqref{Invariants}$ for each are different. However, as we mentioned in Section $\ref{sec:Intro}$,
we still can not conclude that there is no isomorphism between their cubic forms because the values of cubic forms may change if we use different basis of $H^2(M,\Z)$.
In order to solve this problem, we secondly compute the following {\emph{$\lambda$-invariant}} 
to distinguish doubling Calabi-Yau $3$-folds with Picard number two up to diffeomorphism.

Let $M$ be a compact complex $3$-fold with $h^2(M)=2$ whose second Chern class $c_2(M)$ is not zero in
\[
H^4(M,\Z)_f=H^4(M,\Z)/H^4(M,\Z)_t
\]
where $H^4(M,\Z)_t$ denotes the torsion part of $H^4(M,\Z)$. Then the subgroup
\begin{equation}\label{def:SG}
\Set{\ell \in H^2(M,\Z)_f | c_2(M)\cdot \ell =0}
\end{equation}
of $H^2(M,\Z)_f$ is generated by a single element and we denote it by $m$. In \cite{Lee20}, Lee defined the $\lambda$-invariant of $M$ by
\[
\lambda(M):=\norm{m^3}
\]
in order to distinguish whether given two compact complex $3$-folds $M$ and $M'$ are homeomorphic (i.e. diffeomorphic) or not.
Namely we deduce the following.
\begin{proposition}[\cite{Lee20}, Section $8$]\label{prop:lambda_inv}
Let $M$ and $M'$ be two Calabi-Yau $3$-folds with $h^2(M)=h^2(M')=2$.
We assume that their second Chern classes are not zero. If the values of the $\lambda$-invariants for $M$ and $M'$ are different, then the cubic forms of $M$ and $M'$ are not isomorphic.
In particular, $M$ and $M'$ are not diffeomorphic (or equivalently, non-homeomorphic).
\end{proposition}
We will give explicit computations on these invariants in the following section. 
In order to describe two-dimensional cohomology, we apply Theorem $\ref{thm:NamHoon}$. 

\section{Proof of Theorem $\ref{thm:main}$}\label{sec:CubicForm}
As listed in Table $\ref{table:CY3Pic2}$, there are $8$ doubling Calabi-Yau $3$-folds $M$ with Picard number two whose Hodge numbers $(h^{1,1}(M), h^{2,1}(M))$ are overlapped with 
another one. Specifically, these $8$ overlapping Hodge numbers are divided into three cases; $(h^{1,1}(M), h^{2,1}(M))=(2,86)$, $(2,58)$ and $(2,44)$.
Hence our proof is also divided into three subsections from Section $\ref{sec:(2,86)}$ to Section $\ref{sec:(2,44)}$, and we shall find the values of cubic forms and the $\lambda$-invariants of each doubling Calabi-Yau $3$-folds
in order to prove Theorem $\ref{thm:main}$. We refer the reader \cite{Y22} for the detailed calculations of the cubic forms and the $\lambda$-invariants for the abridged doubling Calabi-Yau $3$-folds.

\subsection{$(h^{1,1}(M), h^{2,1}(M))=(2,86)$ case}\label{sec:(2,86)}
To begin with, we consider the following Fano $3$-folds with Picard number one:
\begin{enumerate}
\item[(a)] a quartic hypersurface in $\C P^4$ (ID $1$-$2$), and
\item[(b)] the projective space $\C P^3$ (ID $1$-$17$).
\end{enumerate}
Then the corresponding doubling Calabi-Yau $3$-folds are obtained as follows.

\vskip 4pt

\noindent (a). Let $V:=V(4)\subset \C P^4$ be a smooth quartic hypersurface, $D\in |-K_V| $ a smooth
anicanonical $K3$ divisor. We take a smooth curve $C$ in $D$ representing the self-intersection class of $D\cdot D$.
Let $\varphi: Y\dasharrow V$ be the blow-up of $V$ along the curve $C$.
Taking the proper transform of $D$ under the blow-up $\varphi$, we still denote it by $D$.
By applying Theorem $\ref{thm:DY}$, we obtain the doubling Calabi-Yau $3$-fold $M$ with $(h^{1,1},h^{2,1})=(2,86)$.

\vskip 4pt

\noindent (b). Starting from $\C P^3$ with a smooth anicanonical $K3$ divisor $D\in |\mathcal O_{\C P^3}(4)|$,
we consider a smooth curve $C$ in $D$ as above.
By repeating the same procedure in (a), we obtain the doubling Calabi-Yau $3$-fold $M'$ having the same Hodge numbers $(h^{1,1},h^{2,1})=(2,86)$.

\vskip 4pt

The main part of this section is to show the following statement (see Table $\ref{table:invariants}$).
\begin{proposition}
Let $M$ and $M'$ be as above. Let $f_i$ and $e_i$ $(i=1,2)$ be the generators of $H^2(M,\Z)$ and $H^2(M', \Z)$ respectively. 
Then the cubic products of each doubling Calabi-Yau $3$-fold are:
\begin{align}
\begin{split}\label{eq:CubicForms}
f_1^3&=8, \qquad f_1^2f_2=4, \qquad f_1f_2^2=0, \qquad f_2^3=0,  \qquad\,\,\, \lambda(M)=540  \qquad \text{and} \\
e_1^3&=2, \qquad e_1^2e_2=4, \qquad e_1e_2^2=0, \qquad e_2^3=0,   \qquad \lambda(M')=4320. 
\end{split}
\end{align}
In particular, $M$ and $M'$ are not diffeomorphic.
\end{proposition}
\begin{proof} 
According to Proposition $\ref{prop:lambda_inv}$,
we deduce the second statement from $\eqref{eq:CubicForms}$.
In the following subsections, we shall find the values of invariants in $\eqref{eq:CubicForms}$ by explicit computations.
\end{proof}

\subsubsection{ID $1$-$17$: $\C P^3$ case}\label{subsec:No.17}
 Let $V$ be the $3$-dimensional projective space. Let $D\in |\mathcal O_V(4)|$ be a smooth quartic $K3$ divisor and 
$C$ a smooth curve in $D$ representing $D\cdot D$. Taking $Y_i$ to be the blow-ups $\mathrm{Bl}_C(V)$ of $V$ along $C$ for $i=1,2$, we denote the exceptional divisors $E_i:=\pi_i^{-1}(C)$ 
for $\pi_i: Y_i\dasharrow V$. Then the cohomology ring $H^*(Y_i)=H^*(Y_i,\C)$ can be computed as
\[
H^2(Y_i)=\pi_i^*H^2(V)\oplus \C\bracket{E_i}=\C\bracket{H_i, E_i}
\]
with $H_i=\pi_i^*(H)\subset Y_i$ for the ample generator $H\in H^2(V,\mathbb Z)$ (see \cite[p.$621$]{GH94}).
A straightforward computation shows that the proper transform $\widetilde D_i$ of $D$ in $Y_i$ is $4H_i-E_i$ for each $i$. 
Let $\delta=\bracket{-\widetilde D_1, \widetilde D_2}=\bracket{E_1-4H_1, 4H_2-E_2}$. Here and hereafter we denote the proper transforms $\widetilde D_i$ by $D$ which can be regarded as a divisor in $Y=Y_1\cup Y_2$. 
Since $H^2(Y_i,\Z)=\bracket{H_i, E_i}$ for $i=1,2$, 
any element $\ell_i \in H^2(Y_i,\Z)$ is written as
\[
\ell_1=aH_1+bE_1, \qquad \text{and} \qquad   \ell_2=cE_2+dH_2 \qquad (a,b,c,d \in \Z).
\]
Thus the condition 
$i_1^*(\ell_1)=i_2^*(\ell_2)$ for the inclusion map $i_\alpha: D\to Y_\alpha$ $(\alpha=1,2)$
implies that $a+4b=4c+d$ which means $d=a+4b-4c$.
Then we see that any element in 
$H^2(Y_1, \Z)\times H^2(Y_2, \Z)$ can be expressed as
\[
(aH_1+bE_1, cE_2+(a+4b-4c)H_2 )=(a+4b)(H_1, H_2)-(b+c)(4H_1-E_1, 0)-c\delta.
\]
We remark that the condition in Theorem $\ref{thm:NamHoon}$ has been checked in \cite[p.$215$]{Lee20}. Hence $\eqref{eq:2ndCohomology}$ yields that
\[
H^2(M',\Z)\cong \bracket{(H_1, H_2), (4H_1-E_1, 0)}
\]
up to torsion, where $M'$ is the doubling Calabi-Yau $3$-fold. 
Thus we can take $e_1=(H_1, H_2)$ and $e_2=(4H_1-E_1, 0)$ as generators of $H^2(M',\Z)$.

Next we shall calculate the cubic products in $H^{2*}(M',\Z)$.
Let $L$ be a fiber over a point on $C$ under the blow-up $\pi_1$. Then
\begin{equation}\label{eq:Intersection1}
H_1^3=1, \quad H_1L=0, \quad E_1L=-1 \quad \text{and} \quad H_1^2E_1=0
\end{equation}
by \cite[p.$620$]{GH94}. Since a hyperplane in $\C P^3$ meets $C$ in $16$ points, we find $H_1E_1=16L$. Moreover
\begin{align}\label{eq:Intersection2}
 E_1^2&=-16H_1^2+(4\cdot 16+(-K_{\C P^3}^3))L=-16H_1^2+128L
\end{align}
by the table in \cite[p.$623$]{GH94}. Then $\eqref{eq:Intersection1}$ and $\eqref{eq:Intersection2}$ imply that
\begin{align*}
H_1E_1^2&=-16H_1^3+128H_1L=-16 \qquad \quad \text{and} \\
E_1^3&=-16H_1^2E_1+128E_1L=-16\cdot 0+128\cdot (-1)=-128.
\end{align*}
As one can see in \cite[Chapter IV]{LeeThesis}, the cup product is given by the rule
\begin{align*}
H^2(M',\Z)\times H^2(M',\Z)\times H^2(M',\Z)& \longrightarrow \Z \\
((l_1, l_2), (m_1, m_2), (n_1, n_2))& \longmapsto l_1m_1n_1+l_2m_2n_2.
\end{align*}
Consequently we see that
\begin{align}
\begin{split}\label{eq:cubic1}
 e_1^3&=(H_1, H_2)^3=H_1^3+H_2^3=1+1=2, \\ 
e_1^2e_2&=(H_1, H_2)^2(4H_1-E_1,0)=H_1^2(4H_1-E_1)=4H_1^3-H_1^2 E_1=4, \\
e_1e_2^2&=(H_1, H_2)(4H_1-E_1,0)^2=H_1(4H_1-E_1)^2=H_1(16H_1^2-8H_1E_1+E_1^2) \\
&=16H_1^3+H_1E_1^2=16-16=0, \\
e_2^3&=(4H_1-E_1)^3=64H_1^3-48H_1^2E_1+12H_1E_1^2-E_1^3  \\
&=64-0+12\cdot(-16)+128=0.
\end{split}
\end{align}

Now we compute the $\lambda$-invariant. For this we first need to find the Chern classes as follows.
According to Section $7$ in \cite{Lee10}, let us denote $c_2(M')=(c_2(Y_1), c_2(Y_2))$. By the blow-up formula of Chern classes \cite[p.$610$]{GH94}, we see that
\begin{align*}
c_2(Y_i)&=\pi_i^*(c_2(\C P^3)+\eta_C)-\pi_i^*c_1(\C P^3)\cdot E_i\\
&=(6H_i^2+16H_i^2)-4H_iE_i=22H_i^2-4H_iE_i
\end{align*}
for $i=1,2$, where $\eta_C\in H^4(Y_i,\Z)$ is the class of the blow-up center $C$.
Then the products of $c_2(M')$ and $e_i$ $(i=1,2)$ are
\begin{align*}
e_1\cdot c_2(M')&=H_1\cdot c_2(Y_1)+H_2\cdot c_2(Y_2) \\
&=22H_1^3-4H_1^2E_1+22H_2^3-4H_2^2E_2=44=4\cdot 11 & \text{and}\\
e_2\cdot c_2(M')&=(4H_1-E_1)\cdot c_2(Y_1)+0=88H_1^3+4H_1E_1^2=24=4\cdot6.
\end{align*}
Hence the subgroup 
\[
\Set{e\in H^2(M',\Z)_f | e\cdot c_2(M')=0}
\]
defined in $\eqref{def:SG}$ is generated by a unique element:
\[
\Set{e\in \bracket{e_1,e_2}| e\cdot c_2(M')=0}=\bracket{6e_1-11e_2}.
\]
Consequently we find that
\[
\lambda(M')=\norm{(6e_1-11e_2)^3}=\norm{216e_1^3-1188e_1^2e_2}=4320
\]
by the results in $\eqref{eq:cubic1}$.

\subsubsection{ID $1$-$2$: $V(4)\subset \C P^4$ case}\label{subsec:No.2}
 Let $V$ be a quartic hypersurface in $\C P^4$. Note that $V$ is the Fano $3$-fold with 
 $-K_V^3=4$ (see \cite[p.$215$]{IsPr}). By Lefschetz Hyperplane Theorem, we have more specific description of $V$
such as
\[
h^{p,q}(V)=\begin{array}{ccccccc}
&&& 1 &&& \\
&& 0 && 0 && \\
&0 && 1 && 0& \\
0 && 30 && 30 && 0 \\
&0 && 1 && 0& \\
&& 0 && 0 && \\
&&& 1 &&& \\
\end{array},
\qquad\quad g=g(V)=\frac{H^3}{2}+1=\frac{-K_V^3}{2}+1=3
\]
where $g$ denotes the genus of Fano variety. In particular, $H^3=4$ for the ample generator $H\in H^2(V,\Z)$.
Let $D\in |-K_V|$ be a smooth anticanonical divisor and let $C\in |\mathcal O_D(1)|$ be a smooth curve in $D$ which represents the intersection class of $D\cdot D$.
Then the degree of $C$ is $2g-2$ and this is the reason why $g=\frac{-K_V^3}{2}+1$ is called the {\emph{genus}} of a Fano $3$-fold as in \cite[p.$32$]{IsPr}.
Taking $Y_i$ to be the blow-ups $\mathrm{Bl}_C(V)$ of $V$ along $C$, we again denote the exceptional divisors by $E_i$ for $i=1,2$.
Then the cohomology rings of $Y_i$ are
\[
H^2(Y_i)=\C\bracket{\pi_i^*(H), E_i}=\C\bracket{H_i, E_i}
\]
and the proper transforms $D_i$ of $D$ in $Y_i$ are $H_i-E_i$. Let $\delta=\bracket{-D_1, D_2}=\bracket{E_1-H_1, H_2-E_2}$.
Then we see that any element in $H^2(Y_1,\Z)\times H^2(Y_2,\Z)$ is written as
\[
(aH_1+bE_1, cE_2+(a+b-c)H_2 )=(a+b)(H_1, H_2)-(b+c)(H_1-E_1, 0)-c\delta.
\]
Thus we conclude that
\[
H^2(M,\Z)\cong \bracket{(H_1, H_2), (H_1-E_1, 0)}
\]
up to torsion. Hence in this case, we take $f_1=(H_1, H_2)$ and $f_2=(H_1-E_1, 0)$ as generators of $H^2(M,\Z)$. 

Now we compute the cubic products of $f_i$ in $H^{6}(M,\Z)$. Let us denote by $\pi_i:Y_i=\mathrm{Bl}_C(V)\dasharrow V$ two copies of the blow-ups of $V$ along $C$ for $i=1,2$.
Let $L$ be a fiber over a point on $C$ under the blow-up $\pi_i$. 
Since the intersection number is preserved by the total transform,
we see that $H_i^3=(\pi_i^*H)^3=H^3=4$. 
Moreover, $H_iL=0$ and $E_iL=-1$. 
Let $d$ be the degree of $C$. Since a hyperplane in $V$ will intersect $C$ in $d$ points, its inverse image $H_i$ in $Y_i$ will meet the exceptional divisor $E_i$ in $d$ fibers.
Thus
\[
H_iE_i=dL=(2g-2)L=4L \qquad \text{and} \qquad E_i^2=-4H_i^2+8L.
\]
Then we see that
\begin{align*}
H_i^2E_i&=4H_iL=0, \qquad H_iE_i^2=4E_iL=-4 & \text{and} \qquad \\
E_i^3&=-4H_i^2E_i+8LE_i=-8.
\end{align*}
In sum, we find the following table of the multiplication of the intersection forms on $H^{2*}(Y_i,\Z)$:
\begin{center}
\begin{tabular}{c|ccc}
& $H_i^2$ & $L$  & \\  \hline  \vspace{-0.4cm}\\
$H_i$ & $4$ & $0$ \\
$E_i$ & $0$ & $-1$ \\
\end{tabular}\hfil \quad
\begin{tabular}{c|ccc}
& $H_i$ & $E_i$  & \\ \hline  \vspace{-0.35cm}\\
$H_i$ & $H_i^2$ & $4L$ \\ 
$E_i$ & $4L$ & $-4H_i+8L$ \\
\end{tabular}
\end{center}
Plugging these values into the products, we find that
\begin{align*}
 f_1^3&=(H_1, H_2)^3=H_1^3+H_2^3=8, \\ 
f_1^2f_2&=(H_1, H_2)^2(H_1-E_1,0)=H_1^3-H_1^2 E_1=4, \\
f_1f_2^2&=(H_1, H_2)(H_1-E_1,0)^2=H_1^3-2H_1^2E_1+H_1E_1^2=4-4=0,\\
f_2^3&=(H_1-E_1, 0)^3=H_1^3-3H_1^2E_1+3 H_1E_1^2-E_1^3 =4-0+3\cdot(-4)-(-8)=0.
\end{align*}

Next we calculate the $\lambda$-invariant of the resulting doubling Calabi-Yau $3$-fold $M$. Since $V$ is a degree $4$ smooth hypersurface in $\C P^4$, the total Chern classes of $V$ are given by the formula
\[
\frac{(1+H)^5}{(1+4H)}=(1+5H+10H^2)(1-4H+16H^2)+O(H^3)=1+H+6H^2+O(H^3).
\]
Hence we find that the second Chern classes of $Y_i$ are given by
\begin{equation}\label{eq:2ndChern}
c_2(Y_i)=\pi_i^*(c_2(V)+\eta_C)-\pi_i^*(c_1(V))\cdot E_i=7H_i^2-H_iE_i
\end{equation}
by \cite[p.$610$]{GH94}, where $\eta_C$ denotes the class of the blow-up center $C\in |\mathcal O_D(1)|$.
Then the products of $c_2(M)$ and $f_i$ $(i=1,2)$ are
\begin{align*}
f_1\cdot c_2(M)&=7H_1^3-H_1^2E_1+7H_2^3-H_2^2E_2=56=8\cdot 7,\\
f_2\cdot c_2(M)&=(7H_1^2-H_1E_1)(H_1-E_1)\\
&=7H_1^3-H_1^2E_1-7H_1^2E_1+H_1E_1^2\\
&=7\cdot4-4=24=8\cdot 3.
\end{align*}
Since the subgroup $\Set{f\in \bracket{f_1,f_2} | f\cdot c_2(M)=0}$ of $H^2(M,\Z)$ is generated by a single element
$3f_1-7f_2$, the $\lambda$-invariant of $M$ is 
\begin{align*}
\lambda(M)&=|(3f_1-7f_2)^3|=|27f_1^3-189f_1^2f_2+441f_1f_2^2-343f_2^3|\\
&=|27\cdot 8-189\cdot 4|=540.
\end{align*}

\subsection{$(h^{1,1}(M), h^{2,1}(M))=(2,58)$ case}\label{sec:(2,58)}
In this case, the corresponding doubling Calabi-Yau $3$-folds are listed in Table $\ref{table:CY3Pic2}$ with the underlying Fano $3$-folds, (a) ID $1$-$4$, 
(b) ID $1$-$12$ and (c) ID $1$-$14$. These Fano $3$-folds are described as follows:
\begin{enumerate} 
\item[(a)] a complete intersection of three quadrics in $\C P^6$; $V(2,2,2)\subset \C P^6$,
\item[(b)] a hypersurface of degree $4$ in the weighted projective space $\C P(1,1,1,1,2)$;\\ $V(4)\subset \C P^4(1^4,2)$, and
\item[(c)] a complete intersection of two quadrics in $\C P^5$; $V(2,2)\subset \C P^5$.
\end{enumerate}
In this paper, we only see the detailed computation on the doubling Calabi-Yau $3$-fold with the underlying Fano $3$-fold (a) ID $1$-$4$. We refer the reader to \cite{Y22} for more details of the remaining doubling Calabi-Yau $3$-folds.

\subsubsection{ID $1$-$4$: $V(2,2,2) \subset \C P^{6}$ case}\label{sec:ID1-4}
Let $V$ be a smooth complete intersection of $3$ quadrics in $\C P^6$, which is the Fano $3$-fold with $-K_V^3=8$ and
\[
h^{p,q}(V)=\begin{array}{ccccccc}
&&& 1 &&& \\
&& 0 && 0 && \\
&0 && 1 && 0& \\
0 && 14 && 14 && 0 \\
&0 && 1 && 0& \\
&& 0 && 0 && \\
&&& 1 &&& \\
\end{array}.
\]
By the adjunction formula, we see that
\begin{align*}
K_{V\!(2)}&\cong \left( K_{\C P^6}+[V(2)]\right)\!\big|_{V\!(2)}=-6H, \quad K_{V\!(2,2)}\cong \left( K_{V(2)}+[V(2,2)]\right)\!\big|_{V\!(2,2)}=-3H, \\
  \text{and} \qquad K_{V}&\cong \left( K_{V\!(2,2)}+[V]\right)\!\big|_{V}=(-3+2)H=-H
\end{align*}
where $H\in H(V,\Z)$ is the ample generator and $V(2)\subset \C P^6$ is a smooth quadric hypersurface in $\C P^6$.
Let $D=H\in |-K_V|$ be an anticanonical divisor and $C\in |\mathcal O_D(1)|$ a smooth curve in $D$ representing the intersection class of $D\cdot D$.
For $i=1,2$, we take the blow-ups $Y_i=\mathrm{Bl}_C(V)$ which have the cohomology rings $H^2(Y_i)=\C\bracket{H_i, E_i}$.
Then the proper transforms $D_i$ of $D$ in $Y_i$ are $H_i-E_i$. Thus we set $\delta$ by $\bracket{-D_1, D_2}=\bracket{E_1-H_1, H_2-E_2}$.
We observe that any element in $H^2(Y_1,\Z)\times H^2(Y_2,\Z)$ is written as
\[
(aH_1+bE_1, cE_2+(a+b-c)H_2 )=(a+b)(H_1, H_2)-(b+c)(H_1-E_1, 0)-c\delta.
\]
Consequently, we find that
\[
H^2(M,\Z)\cong \bracket{(H_1, H_2), (H_1-E_1, 0)}
\]
up to torsion. This implies that two generators of $H^2(M,\Z)$ can be taken as $e_1=(H_1, H_2)$ and $e_2=(H_1-E_1, 0)$.

In order to compute the cubic forms in $H^{6}(M,\Z)$, we first see that the Fano genus $g$ of $V$ is
\[
g=\frac{-K_V^3}{2}+1=\frac{8}{2}+1=5.
\]
Then the straightforward computation shows that $H_i^3=8$, $H_iL=0$ and $E_iL=-1$ where $L$ is a fiber over a point on $C$ under the blow-up.
Furthermore, for $d=\deg C$, we have 
\begin{align*}
H_iE_i&=dL=(2g-2)L=8L \qquad \text{and}\\
H_i^2E_i&=H_i(H_iE_i)=8H_iL=0.
\end{align*}
Let $\tau=2g$ be the number of branches of the double curve $Y_i \supset \widetilde C \stackrel{2:1}{\longrightarrow} C\subset V$. By the list in \cite[p.$623$]{GH94}, we see that
\begin{align*}
E_i^2&=-dH_i^2+(4d+2g-2-2\tau)L=-8H_i^2+(32+10-2-20)L=-8H_i^2+20L, \\
H_iE_i^2&=H_i(-8H_i^2+20L)=-8H_i^3+20H_iL=-8\cdot 8=-64, \qquad \text{and}  \\
E_i^3&=E_i(-8H_i^2+20L)=-8E_iH_i^2+20E_iL=-20.
\end{align*}
In the following table, we summarize the values of the multiplication of the intersection forms on $H^{2*}(Y_i,\Z)$:
\begin{center}
\begin{tabular}{c|ccc}
& $H_i^2$ & $L$  &  \\ \hline  \vspace{-0.4cm}\\
$H_i$ & $8$ & $0$ \\
$E_i$ & $0$ & $-1$ \\
\end{tabular}\hfil \quad
\begin{tabular}{c|ccc}
& $H_i$ & $E_i$  &  \\ \hline  \vspace{-0.35cm}\\
$H_i$ & $H_i^2$ & $8L$ \\ 
$E_i$ & $8L$ & $-8H_i^2+20L$ \\
\end{tabular}
\end{center}
Substituting these values into the cubic forms, we find that
\begin{align*}
 e_1^3&=(H_1, H_2)^3=H_1^3+H_2^3=16, \\ 
e_1^2e_2&=(H_1, H_2)^2(H_1-E_1,0)=H_1^3-H_1^2 E_1=8, \\
e_1e_2^2&=(H_1, H_2)(H_1-E_1,0)^2=H_1^3-2H_1^2E_1+H_1E_1^2=8-64=-56,\\
e_2^3&=(H_1-E_1,0)^3=H_1^3-3H_1^2E_1+3 H_1E_1^2-E_1^3 =8+3\cdot(-64)-(-20)=-164.
\end{align*}

Next we compute the $\lambda$-invariant. Since $V$ is a complete intersection of two quadrics in $\C P^5$, the total Chern classes of $V$ are given by the formula
\begin{align*}
\frac{(1+H)^7}{(1+2H)^3}&=\bigl(1+7H+\begin{pmatrix} 7 \\2 \end{pmatrix}H^2 \bigr)\bigl(1+2H \bigr)^{-3}+O(H^3) \\
&=(1+7H+21H^2)(1-6H+24H^2)+O(H^3)=1+H+3H^2+O(H^3).
\end{align*}
Hence the second Chern classes of $Y_i$ are computed as
\begin{align*}
c_2(Y_i)&=\pi_i^*(c_2(V)+\eta_C)-\pi_i^*(c_1(V))\cdot E_i\\
&=\pi_i^*(3H^2+H^2)-H_iE_i=4H_i^2-H_iE_i.
\end{align*}
Then the products of $c_2(M)$ and $e_i$ are given by
\begin{align*}
e_1\cdot c_2(M)&=4H_1^3-H_1^2E_1+4H_2^3-H_2^2E_2=64=32\cdot 2,\\
e_2\cdot c_2(M)&=(H_1-E_1)(4H_1^2-H_1E_1)\\
&=4H_1^3-H_1^2E_1-4H_1^2E_1+H_1E_1^2\\
&=4\cdot 8-64=-32=32\cdot (-1).
\end{align*}
Since the subgroup $\Set{e\in \bracket{e_1,e_2} | e\cdot c_2(M)=0}$ of $H^2(M,\Z)$ is generated by a single element
$e_1+2e_2$, the $\lambda$-invariant of $M$ is 
\begin{align*}
\lambda(M)&=|(e_1+2e_2)^3|=|e_1^3+6e_1^2e_2+12 e_1e_2^2+8e_2^3|\\
&=|16+48-672-1312|=1920.
\end{align*}

\subsubsection{ID $1$-$12$: $V(4) \subset \C P^{4}(1^4,2)$ case}\label{sec:ID1-12}
This and the following subsection collects the minimum amount of calculation necessary to see the values of the cubic forms and the $\lambda$-invariants.

Let $V=V(4)\subset \C P^4(1^4,2)$ be a smooth hypersurface of degree $4$ in the weighted projective space. As usual,
we set $D\in |\mathcal O_V(2)|$, $C\in |\mathcal O_D(2)|$ and $\pi_i:Y_i=\mathrm{Bl}_C(V)\dasharrow V$ for $i=1,2$.
Then we see that the proper transform $D_i$ of $D$ in $Y_i$ is $2H_i-E_i$ and $H^2(Y_i)=\C\braket{H_i, E_i}$ for each $i$.
Thus any element in $H^2(Y_1,\Z)\times H^2(Y_2,\Z)$ can be written as
\[
(a+2b)(H_1, H_2)-(b+c)(2H_1-E_1, 0)-c\delta, \qquad \delta=\braket{E_1-2H_1, 2H_2-E_2}.
\]
This implies that
\[
H^2(M,\Z)\cong \braket{(H_1,H_2),(2H_1-E_1,0)}
\]
up to torsion. Setting $e_1=(H_1,H_2)$ and $e_2=(2H_1-E_1,0)$ as generators of $H^i(M,\Z)$, we find that
\begin{align*}
 e_1^3&=(H_1, H_2)^3=H_1^3+H_2^3=32, \\ 
e_1^2e_2&=(H_1, H_2)^2(2H_1-E_1,0)=2H_1^3-H_1^2 E_1=32, \\
e_1e_2^2&=(H_1, H_2)(2H_1-E_1,0)^2=4H_1^3-4H_1^2E_1+H_1E_1^2=-192,\\
e_2^3&=(2H_1-E_1,0)^3=8H_1^3-12H_1^2E_1+6 H_1E_1^2-E_1^3 =-1364.
\end{align*}
In the same manner as the previous calculation in Section $\ref{sec:ID1-4}$, the second Chern class of $Y_i$ is $c_2(Y_i)=10H_i^2-2H_iE_i$ for each $i$.
Consequently, the subgroup $\set{e\in\braket{e_1,e_2}| e\cdot c_2(M)=0}$ of $H^2(M,\Z)$ is generated by $3e_1+5e_2$. Hence we conclude that the $\lambda$-invariant is $\lambda(M)=|(3e_1+5e_2)^3|=208516$.

\subsubsection{ID $1$-$14$: $V(2,2) \subset \C P^{5}$ case}\label{sec:ID1-14}
Let $V=V(2,2)\subset \C P^5$ be a smooth complete intersection of two quadrics in $\C P^5$. 
We set $D\in |\mathcal O_V(2)|$, $C\in |\mathcal O_D(2)|$ and $\pi_i:Y_i=\mathrm{Bl}_C(V)\dasharrow V$ for $i=1,2$.
Then we see that the proper transform $D_i$ of $D$ in $Y_i$ is $2H_i-E_i$ and $H^2(Y_i)=\C\braket{H_i, E_i}$ for each $i$.
Thus any element in $H^2(Y_1,\Z)\times H^2(Y_2,\Z)$ can be written as
\[
(a+2b)(H_1, H_2)-(b+c)(2H_1-E_1, 0)-c\delta, \qquad \delta=\braket{E_1-2H_1, 2H_2-E_2}.
\]
This implies that
\[
H^2(M,\Z)\cong \braket{(H_1,H_2),(2H_1-E_1,0)}
\]
up to torsion. Setting $e_1=(H_1,H_2)$ and $e_2=(2H_1-E_1,0)$ as generators of $H^i(M,\Z)$, we find that
\begin{align*}
 e_1^3&=(H_1, H_2)^3=H_1^3+H_2^3=64, \\ 
e_1^2e_2&=(H_1, H_2)^2(2H_1-E_1,0)=2H_1^3-H_1^2 E_1=64, \\
e_1e_2^2&=(H_1, H_2)(2H_1-E_1,0)^2=4H_1^3-4H_1^2E_1+H_1E_1^2=-896,\\
e_2^3&=(2H_1-E_1,0)^3=8H_1^3-12H_1^2E_1+6 H_1E_1^2-E_1^3 =-5796.
\end{align*}
Then the second Chern class of $Y_i$ is $c_2(Y_i)=7H_i^2-2H_iE_i$ for each $i$.
As a consequence, the subgroup $\set{e\in\braket{e_1,e_2}| e\cdot c_2(M)=0}$ of $H^2(M,\Z)$ is generated by $25e_1+7e_2$. Hence we conclude that the $\lambda$-invariant is $\lambda(M)=|(25e_1+7e_2)^3|=3440828$.

\subsection{$(h^{1,1}(M), h^{2,1}(M))=(2,44)$ case}\label{sec:(2,44)}
Now we consider the case where the doubling Calabi-Yau $3$-folds have the same Hodge numbers $(h^{1,1}(M), h^{2,1}(M))=(2,44)$, that is, the underlying Fano $3$-folds, (a) ID $1$-$8$, 
(b) ID $1$-$9$ and (c) ID $1$-$10$. 
We remark that these Fano $3$-folds have the following geometric description:
\begin{enumerate}
\item[(a)] a section of Pl\"ucker embedding of $\mathrm{SGr}(3,6)$ by codimension $3$ subspace, where $\mathrm{SGr}(3,6)$ is the Lagrangian Grassmannian; 
$V(1,1,1)\hookrightarrow \mathrm{SGr}(3,6)$,

\item[(b)] a section of $\mathrm{G}_2\mathrm{Gr}(2,7)$ by codimension $2$ subspace; $V(1,1)\hookrightarrow \mathrm{G}_2\mathrm{Gr}(2,7)$,  and

\item[(c)] the zero locus of $\left( \bigwedge^{\!2}\mathcal V^{\vee} \right)^{\!\oplus 3}$ on $\mathrm{Gr}(3,7)$ where $\mathcal V\to \mathrm{Gr}(3,7)$ is the tautological rank $3$ vector bundle over the Grassmannian $\mathrm{Gr}(3,7)$.
\end{enumerate}
In the above description (b), $\mathrm{G}_2\mathrm{Gr}(2,7)$ denotes the adjoint $\mathrm G_2$-Grassmannian which is the zero locus of the section $s\in \bigwedge^{\!3} \C^7$ corresponding to the $\mathrm G_2$-invariant $3$-form. 
See \cite{Fano}, \cite[Chapter $4$]{IsPr}, \cite[Section 5]{D08} for more details. Systematically, all of these Fano $3$-folds are expressed as anticanonically embedded Fano $3$-folds $V=V_{2g-2}\subset \C P^{g+1}$
with Picard number $1$ and genus $g$.
Moreover, we may assume that $\mathrm{Pic}(V)=H\cdot \Z$ where $H$ is the unique generator of 
$H^2(V,\Z)$ and $H=-K_V$ for each case (a) $g=9: V_{16}\subset \C P^{10}$, (b) $g=10: V_{18}\subset \C P^{11}$
and (c) $g=12: V_{22}\subset \C P^{13}$, respectively.

\subsubsection{ID $1$-$8$: $V_{16}\subset \C P^{10}$ case}\label{sec:No.8}
Firstly, we consider case (a). Let $V=V_{16}\subset \C P^{10}$ be an anticanonically embedded Fano $3$-fold with genus $g=9$, $\mathrm{Pic}(V)=\Z\cdot H$ and $-K_V=H$.
According to \cite{Fano}, we have $-K_V^3=16$ and 
\begin{equation}\label{eq:No.9Hodge}
h^{p,q}(V)=\begin{array}{ccccccc}
&&& 1 &&& \\
&& 0 && 0 && \\
&0 && 1 && 0& \\
0 && 3 && 3 && 0 \\
&0 && 1 && 0& \\
&& 0 && 0 && \\
&&& 1 &&& \\
\end{array}.
\end{equation}
Let $D\in |\mathcal O_V(1)|$ be an anticanonical divisor and $C\in |\mathcal O_D(1)|$ a smooth curve in $D$. 
Setting $Y_i$ to be two copies of the blow-up $\mathrm{Bl}_C(V)$ for $i=1,2$, we see that $H^2(Y_i)=\C\bracket{H_i, E_i}$ and $H^2(M,\Z)\cong \bracket{(H_1,H_2), (H_1-E_1,0)}$
up to torsion. This yields that generators of $H^2(M,\Z)$ are given by $e_1=(H_1, H_2)$ and $e_2=(H_1-E_1, 0)$. 
Then we find that $H_i^3=16$, $H_iL=0$ and $E_iL=-1$ where $L$ is a fiber over a point on $C$ under the blow-up. Moreover, for $d=\deg C$, we have
\begin{align*}
H_iE_i&=dL=(2g-2)L=16L \quad \text{and}\quad
H_i^2E_i=H_i(H_iE_i)=16H_iL=0.
\end{align*}

In the same manner as in Section $\ref{sec:(2,58)}$, let us denote the number of branches of the double curve $\widetilde C$ by $\tau$.
Then we find that
\begin{align*}
E_i^2&=-dH_i^2+(4d+2g-2-2\tau)L \\
&=-16H_i^2+(64+18-2-36)L=-16H_i^2+44L,\\
H_iE_i^2&=H_i(-16H_i^2+44L)=-16H_i^3+44H_iL=-16\cdot 16=-256,\\
E_i^3&=E_i(-16H_i^2+44L)=-16E_iH_i^2+44E_iL=-44.
\end{align*}
Consequently, we have the following table of the multiplication of the intersection forms on $H^{2*}(Y_i,\Z)$:
\begin{center}
\begin{tabular}{c|ccc}
& $H_i^2$ & $L$  &  \\ \hline  \vspace{-0.4cm}\\
$H_i$ & $16$ & $0$ \\
$E_i$ & $0$ & $-1$ \\
\end{tabular}\hfil \quad
\begin{tabular}{c|ccc}
& $H_i$ & $E_i$  & \\ \hline  \vspace{-0.35cm}\\
$H_i$ & $H_i^2$ & $16L$ \\ 
$E_i$ & $16L$ & $-16H_i^2+44L$ \\
\end{tabular} 
\vskip 10pt
\end{center}
Substituting these values into the cubic products, we see that
\begin{align*}
 e_1^3&=(H_1, H_2)^3=H_1^3+H_2^3=32, \\ 
e_1^2e_2&=(H_1, H_2)^2(H_1-E_1,0)=H_1^3-H_1^2 E_1=16, \\
e_1e_2^2&=(H_1, H_2)(H_1-E_1,0)^2=H_1^3-2H_1^2E_1+H_1E_1^2=-240,\\
e_2^3&=(H_1-E_1, 0)^3=H_1^3-3H_1^2E_1+3 H_1E_1^2-E_1^3 =-708.
\end{align*}

Next we compute the $\lambda$-invariant of the doubling Calabi-Yau $3$-fold $M$. Since $V=V_{16}\subset \C P^{10}$ is an anticanonically embedded Fano $3$-fold with $-K_V=H$, we see that the first Chern class of $V$ is given by
$c_1(V)=H$. In order to find the second Chern class of $V$, we use the Riemann-Roch-Hirzebruch formula
\begin{equation}\label{thm:RRH}
\sum_{q=0}^n(-1)^q\dim H^q(V,\Omega^p)=\int_V {td} (V)\,  {ch}\!\! \left(\smallwedge^p T^*V \right)
\end{equation}
for $n=3$ and $p=0$. This yields the equality
\begin{align}
\sum_{q=0}^3(-1)^q\dim H^q(V,\Omega^0)&=\int_V\left( 1+\frac{1}{2}c_1(V)+\frac{1}{12}(c_1(V)^2+c_2(V))+\frac{1}{24}c_1(V)c_2(V)\right){ch}\!\! \left(\smallwedge^0 T^*V \right) \notag \\
\Leftrightarrow \quad h^{0,0}- h^{0,1}+ h^{0,2}- h^{0,3}&=\frac{1}{24}\int_V c_1(V)c_2(V) \label{eq:RRH}
\end{align}
Suppose that $c_2(V)=aH^2$ for $a\in \mathbb Q$. Then the Hodge diamond $\eqref{eq:No.9Hodge}$ and the equality $\eqref{eq:RRH}$ imply that
\[
\frac{1}{24}\int_VaH^3=1 \quad \Leftrightarrow \quad a=\frac{3}{2}
\]
by $\int_V H^3=(-K_V^3)=16$. Thus, we find $c_2(V)=\frac{3}{2}H^2$. As we have seen in $\eqref{eq:2ndChern}$, the second Chern classes of $Y_i$ are given by 
\begin{align*}
c_2(Y_i)&=\pi_i^*(c_2(V)+\eta_C)-\pi_i^*(c_1(V))\cdot E_i \\
&=\pi_i^*\left(\frac{3}{2}H^2+H^2 \right)-H_iE_i=\frac{5}{2}H_i^2-H_iE_i.
\end{align*}
Then the products of $c_2(M)$ and $e_i$ are
\begin{align*}
e_1\cdot c_2(M)&=\frac{5}{2}H_1^3-H_1^2E_1+\frac{5}{2}H_2^3-H_2^2E_2=80=8\cdot 10,\\
e_2\cdot c_2(M)&=(H_1-E_1)c_2(Y_1)=(H_1-E_1)\left( \frac{5}{2}H_1^2-H_1E_1\right)\\
&=\frac{5}{2}H_1^3+H_1E_1^2=\frac{5}{2}\cdot 16+(-256)=-216=-8\cdot 27.
\end{align*}
Since the subgroup $\Set{e\in \bracket{e_1,e_2} | e\cdot c_2(M)=0}$ of $H^2(M,\Z)$ is generated by \\
$27e_1+10e_2$, we see that the $\lambda$-invariant of $M$ is given by
\begin{align*}
\lambda(M)&=|(27e_1+10e_2)^3|=|27^3e_1^3+3\cdot27^2\cdot 10\cdot e_1^2e_2+3\cdot 27\cdot 10^2 e_1e_2^2+10^3 e_2^3|=1672224.
\end{align*}

\subsubsection{ID $1$-$9$: $V_{18}\subset \C P^{11}$ case}\label{sec:No.9}
Secondly, we shall consider case (b). Thus we suppose that $V=V_{18}\subset \C P^{11}$, $g=10$, $\mathrm{Pic}(V)=\Z \cdot H$ and $-K_V=H$. Furthermore, we have
$-K_V^3=18$ and 
\begin{equation*}\label{eq:No.8Hodge}
h^{p,q}(V)=\begin{array}{ccccccc}
&&& 1 &&& \\
&& 0 && 0 && \\
&0 && 1 && 0& \\
0 && 2 && 2 && 0 \\
&0 && 1 && 0& \\
&& 0 && 0 && \\
&&& 1 &&& \\
\end{array}.
\end{equation*}
Setting $D\in|\mathcal O_V(1)|$, $C\in |\mathcal O_D(1)|$ and $\pi_i:Y_i=\mathrm{Bl}_C(V)\dasharrow V$ for $i=1,2$, we see that
$H^2(Y_i)=\C\braket{H_i, E_i}$ and $H^2(M,\Z)\cong \braket{(H_1,H_2),(H_1-E_1,0)}$ up to torsion. Hence two generators of $H^2(M,\Z)$ are taken as
$e_1=(H_1,H_2)$ and $e_2=(H_1-E_1,0)$. Consequently, we find the values of the cubic forms as follows:
\begin{align*}
 e_1^3&=(H_1, H_2)^3=H_1^3+H_2^3=36, \\ 
e_1^2e_2&=(H_1, H_2)^2(H_1-E_1,0)=H_1^3-H_1^2 E_1=18, \\
e_1e_2^2&=(H_1, H_2)(H_1-E_1,0)^2=H_1^3-2H_1^2E_1+H_1E_1^2=-306,\\
e_2^3&=(H_1-E_1,0)^3=H_1^3-3H_1^2E_1+3 H_1E_1^2-E_1^3 =-904.
\end{align*}
As we computed in Section $\ref{sec:No.8}$, the second Chern class of $V$ is calculated by the Riemann-Roch-Hirzebruch formula $\eqref{thm:RRH}$, from which we conclude that $c_2(V)=\frac{4}{3}H^2$.
Thus the second Chern classes of $Y_i$ are
\[
c_2(Y_i)=\pi_i^*\left(\frac{4}{3}H^2+H^2  \right)-H_iE_i=\frac{7}{3}H_i^2-H_iE_i
\] 
for $i=1,2$. Then the subgroup $\set{e\in\braket{e_1,e_2}| e\cdot c_2(M)=0}$ of $H^2(M,\Z)$ is generated by $47e_1+14e_2$. This implies that the $\lambda$-invariant is $\lambda(M)=|(47e_1+14e_2)^3|=5529560$.

\subsubsection{ID $1$-$10$: $V_{22}\subset \C P^{13}$ case}
Finally, we consider case (c), that is, $V=V_{22}\subset \C P^{13}$ is an anticanonically embedded Fano $3$-fold with genus $g=12$, $\mathrm{Pic}(V)=\Z \cdot H$ and $-K_V=H$.
Note that the unique such $3$-fold with $\mathrm{Aut}(V)=\mathrm{PGL}(2,\C)$ is called the Mukai-Umemura $3$-fold, and we refer the reader to \cite{D08} and references therein for more details. 
By repeating the same computation in the previous section, 
we find two generators $e_1=(H_1,H_2)$ and $e_2=(H_1-E_1,0)$ of $H^2(M,\Z)$.
Thus we have
\begin{align*}
 e_1^3&=(H_1, H_2)^3=H_1^3+H_2^3=44, \\ 
e_1^2e_2&=(H_1, H_2)^2(H_1-E_1,0)=H_1^3-H_1^2 E_1=22, \\
e_1e_2^2&=(H_1, H_2)(H_1-E_1,0)^2=H_1^3-2H_1^2E_1+H_1E_1^2=-462,\\
e_2^3&=(H_1-E_1,0)^3=H_1^3-3H_1^2E_1+3 H_1E_1^2-E_1^3 =-1358.
\end{align*}
Moreover, the second Chern classes of $Y_i$ for $i=1,2$ are given by
\[
c_2(Y_i)=\pi_i^*\left(\frac{12}{11}H^2+H^2  \right)-H_iE_i=\frac{23}{11}H_i^2-H_iE_i.
\] 
Then the subgroup $\set{e\in\braket{e_1,e_2}| e\cdot c_2(M)=0}$ of $H^2(M,\Z)$ is generated by $219e_1+46e_2$. Hence the $\lambda$-invariant is $\lambda(M)=|(219e_1+46e_2)^3|=122507896$.


\begin{thebibliography}{99}

\bibitem[AP]{MAP} 
\newblock $6$-manifolds: $1$-connected. The Manifold Atlas Project.\\
\newblock {\tt http://www.map.mpim-bonn.mpg.de/6-manifolds:\_1-connected}

\bibitem[Bat94]{Bat94}
V. Batyrev,
\newblock {\em Dual polyhedra and mirror symmetry for Calabi-Yau hypersurfaces in toric varieties},
\newblock J. Alg. Geom. {\bf{3}} ($1994$), $493$--$535$. 

\bibitem[BI15]{BI15}
G. Bini and D. Iacono,
\newblock {\em Diffeomorphism classes of Calabi-Yau varieties},
\newblock Rend. Semin. Mat. Univ. Politec. Torino {\bf{73}} ($2015$), $9$--$20$. 


\bibitem[DY14]{DY14}
M. Doi and N. Yotsutani
\newblock {\em Doubling construction of Calabi-Yau threefolds},
\newblock New York J. Math. {\bf{20}} $(2014)$, $1$--$33$.

\bibitem[DY22]{DY22}
M. Doi and N. Yotsutani
\newblock {\em Differential geometric global smoothings of simple normal crossing complex surfaces with trivial canonical bundle},
\newblock arXiv:$2203.09304$, to appear in {{Complex Manifolds}}.

\bibitem[D08]{D08}
S. Donaldson,
\newblock {\em K\"ahler geometry on toric manifolds, and some other manifolds with large symmetry},
\newblock Handbook of geometric analysis. No. $1$, $29$--$75$, 
\newblock Adv. Lect. Math. (ALM), {\bf{7}}, Int. Press, Somerville, MA, 2008. 


\bibitem[FFR19]{FFR19}
S. Felten, M. Filip and H. Ruddat,
\newblock{\em Smoothing toroidal crossing spaces},
\newblock Forum Math, {\bf{9}}:e7 $1$--$36$ ($2021$).


\bibitem[Fr83]{Fr83}
R. Friedman, 
\newblock {\em Global smoothings of varieties with normal crossings}, 
\newblock Ann. of Math. {\bf{118}} ($1983$), $75$--$114$.

\bibitem[GH94]{GH94}
P. Griffiths and J. Harris,
\newblock {\em Principles of algebraic geometry},
\newblock Wiley Classics Library. John Wiley and Sons, Inc., New York, $1994$. xiv+$813$ pp.

\bibitem[Gr97]{Gr97}
M. Gross,
\newblock {\em The deformation space of Calabi-Yau $n$-folds can be obstructed}, 
\newblock Mirror symmetry I\hspace{-0.1em}I, AMS/IP Stud. Adv. Math., {\bf{1}}, 
\newblock Amer. Math. Soc, Providence, RI, $1997$, $401$-$411$.


\bibitem[HS19]{HS19}
K. Hashimoto and T. Sano, 
\newblock {\em Examples of non-K\"ahler Calabi-Yau $3$-folds with arbitrarily large $b_2$}, 
\newblock arXiv:$1902$.$01027$, to appear in Geom. Topol.

\bibitem[HHN15]{HHN15}
M. Haskins, H-J Hein and J. Nordstr\"om,
\newblock{\em Asymptotically cylindrical Calabi-Yau manifolds},
\newblock J. Diff. Geom {\bf{101}} ($2015$), $213$--$265$.

\bibitem[Huy16]{Huy16}
D. Huybrechts,
\newblock{\em Lectures on $K3$ surfaces},
\newblock Cambridge Studies in Advanced Mathematics, {\bf{158}}. Cambridge University Press, Cambridge, $2016$. xi+pp. $485$.

\bibitem[Is77]{Is77}
V. A. Iskovskih,
\newblock{\em Fano threefolds. I. I\hspace{-0.1em}I},
\newblock Izv. Akad. Nauk SSSR Ser. Mat. {\bf{41}} ($1977$), $516$--$562$, $717$.
 \newblock {\bf{42}} ($1978$), $506$--$549$. 
 \newblock English translation: Math. USSR, Izv., {\bf{11}} ($1977$), $485$--$527$ and {\bf{12}} ($1978$), $469$--$506$.

\bibitem[IsPr99]{IsPr}
V. A. Iskovskikh and Y. Prokhorov,
\newblock {\em Fano varieties}, Algebraic geometry V, 
\newblock Encyclopaedia Math. Sci. Springer, Berlin {\bf{47}}, $1999$. pp. $247$.

\bibitem[J73]{J73}
P. Jupp, 
\newblock {\em Classification of certain $6$-manifolds}, 
\newblock Proc. Camb. Phil. Soc. {\bf{73}} ($1973$), $293$--$300$.

\bibitem[Ka92]{Kaw92}
Y. Kawamata,
\newblock{ \em Unobstructed deformations. A remark on a paper of Z. Ran: "Deformations of manifolds with torsion or negative canonical bundle''}
\newblock [J. Algebraic Geom. {\bf{1}} ($1992$), no. $2$, $279$--$291$;] 
\newblock J. Algebraic Geom. {\bf{1}} ($1992$), no. $2$, $183$--$190$. 


\bibitem[KN94]{KN94}
Y. Kawamata, Y. Namikawa, 
\newblock{\em Logarithmic deformations of normal crossing varieties and smoothing of degenerate Calabi-Yau varieties}, 
\newblock Invent. Math. J. {\bf{118}} ($1994$), $395$--$409$.

\bibitem[KW14]{KW14}
A. Kanazawa and M.H.Wilson,
\newblock {\em Trilinear forms and Chern classes of Calabi-Yau threefolds},
\newblock Osaka J. Math. {\bf{51}} ($2014$), $203$--$213$.

\bibitem[Kov03]{Kov03}
A. Kovalev, 
\newblock {\em Twisted connected sums and special Riemannian holonomy},
\newblock J. Reine Angew. Math. {\bf{565}} ($2003$), $125$--$160$.

\bibitem[KL11]{KL11}
A. Kovalev and N-H. Lee,
\newblock {\em $K3$ surfaces with non-symplectic involution and compact irreducible $G_2$-manifolds},
\newblock Math. Proc. Camb. Phil. Soc. {\bf{151}} $(2011)$, $193$--$218$. 

\bibitem[KS00]{KS00}
M. Kreuzer and H. Skarke,
\newblock {\em Complete classification of reflexive polyhedra in four dimensions},
\newblock Adv. Theor. Math. Phys. {\bf{4}} $(2000)$, $1209$--$1231$.


\bibitem[Lee]{LeeThesis}
N-H. Lee, 
\newblock {\em Constructive Calabi-Yau Manifolds},
\newblock Dissertation, University of Michigan, $2006$. 

\bibitem[Le10]{Lee10}
N-H. Lee, 
\newblock {\em Calabi-Yau construction by smoothing normal crossing varieties},
\newblock Internat. J. Math. {\bf{21}} ($2010$), $701$--$725$. 

\bibitem[Le20]{Lee20}
N-H. Lee, 
\newblock {\em Mirror pairs of Calabi-Yau threefolds from mirror pairs of quasi-Fano threefolds},
\newblock J. Math. Pures Appl. {\bf{141}} ($2020$), $195$--$219$. 

\bibitem[LF]{Fano} 
\newblock List of Fano varieties--fanography.
\newblock {\tt https://www.fanography.info}


\bibitem[MM81]{MM81}
S. Mori and S.  Mukai,
\newblock {\em Classification of Fano $3$-folds with $B_2\geq 2$},
\newblock  Manuscripta Math. {\bf{36}} ($1981$/$1982$), $147$--$162$.

\bibitem[MM03]{MM03}
S. Mori and S.  Mukai,
\newblock {\em Erratum: Classification of Fano $3$-folds with  $B_2\geq 2$},
\newblock Manuscripta Math. {\bf{110}} ($2003$), $407$. 


\bibitem[OkVa95]{OkVa95}
C. Okonek and A. Van de Ven, 
\newblock {\em Cubic forms and complex $3$-folds}, 
\newblock Enseign. Math. {\bf{41}} ($1995$), $297$--$333$.


\bibitem[Ra92]{Ra92}
Z. Ran, 
\newblock {\em Deformations of manifolds with torsion or negative canonical bundle}, 
\newblock J. Algebraic Geom. {\bf{1}} ($1992$), $279$--$291$.


\bibitem[Th99]{Th99}
R. Thomas,
\newblock{\em A holomorphic Casson invariant for Calabi-Yau $3$-folds, and bundles on $K3$ fibrations},
\newblock J. Diff. Geom {\bf{53}} ($1999$), $367$--$438$.


\bibitem[Wa66]{Wa66}
C.T.C. Wall, 
\newblock {\em Classification problems in differential topology V. On certain 6-manifolds}, 
\newblock Inv. Math. {\bf{1}} ($1966$), $355$--$374$.

\bibitem[Y22]{Y22}
N. Yotsutani, 
\newblock {\em Appendix to ``Diffeomorphism classes of the doubling Calabi-Yau threefolds with Picard number two''}, 
\newblock RIMS Kokyuroku {\bf{2226}}, ($2022$), $84$--$94$.  ISSN $1880$-$2818$.





\end{thebibliography}
\end{document}